\documentclass[11pt,a4paper]{article}

\usepackage[margin=30mm]{geometry}
\usepackage[T1]{fontenc}
\usepackage{lmodern}
\usepackage{microtype}
\usepackage{amsmath,amssymb,amsthm,mathtools}
\usepackage{booktabs}
\usepackage{enumitem}
\usepackage[colorlinks=true,linkcolor=blue,citecolor=blue,urlcolor=blue]{hyperref}
\usepackage[nameinlink,capitalise]{cleveref}

\newcommand{\kk}{\Bbbk}
\newcommand{\hK}{\mathbb K}
\newcommand{\ot}{\otimes}
\newcommand{\hot}{\widehat\otimes}
\newcommand{\tri}{\mathbin{\triangleleft}}
\newcommand{\id}{\operatorname{id}}
\newcommand{\ev}{\operatorname{ev}}
\newcommand{\coev}{\operatorname{coev}}
\newcommand{\Writhe}{\operatorname{wr}}
\newcommand{\End}{\operatorname{End}}
\newcommand{\ao}{\mathrm{ao}}
\newcommand{\rc}{\mathrm{rc}}
\newcommand{\cen}{\mathrm{cen}}
\newcommand{\rootpart}{\mathrm{root}}

\newtheorem{theorem}{Theorem}[section]
\newtheorem{proposition}[theorem]{Proposition}
\newtheorem{lemma}[theorem]{Lemma}

\theoremstyle{definition}

\theoremstyle{remark}
\newtheorem{remark}[theorem]{Remark}

\title{The \(V_n\) Invariants as Colored Links--Gould Invariants:\\
A Root--Center Approach}
\author{Jiuhe Liu}
\date{}

\begin{document}

\maketitle

\begin{abstract}
The knot invariants \(V_n\) arise from a rank-two Nichols algebra and a
\(4n\)-dimensional right Yetter--Drinfeld module, whereas the colored
Links--Gould invariants are defined from typical
\(U_q(\mathfrak{sl}(2|1))\)-modules.  We prove that these two constructions
agree, up to the mirror and parameter inversion forced by the right-module
convention.  The proof is structural.  We first realize the braid operator
\(T_n\) as a gauge transform of a canonical super
Yetter--Drinfeld braiding.  We then construct the relevant right--right
paired double, prove that its Hopf pairing is perfect in every root degree,
and obtain its completed universal \(R\)-matrix.  An explicit Abelian
Drinfeld twist separates this double into an all-odd
\(U_\hbar(\mathfrak{sl}(2|1))\) root factor and a commutative central
factor.  Under this factorization, the Nichols module becomes a typical
all-odd highest-weight module tensored with a one-dimensional central
module.  Finally, we transport duality, all four oriented crossings,
partial transposes, and writhe normalization.  For every oriented knot
\(\mathcal K\), every \(n\geq 1\), and every
\(\beta\neq 0,-1\), the result is
\[
 V_{n,\mathcal K}\!\left(Q^{2n\beta+n},Q^2\right)
 =
 LG_{\mathcal K}^{(n)}\!\left(Q^{-n\beta},Q^{-1}\right)
 =
 LG_{\overline{\mathcal K}}^{(n)}\!\left(Q^{n\beta},Q\right).
\]
In particular, the scalar-identity property conjectured for the
endomorphism-valued \(V_n\) construction follows from the simplicity of
the corresponding typical module.
\end{abstract}

\tableofcontents

\section{Introduction}

The invariants \(V_n\) of Garoufalidis and Kashaev are extracted from a
family of braid-form Yang--Baxter operators \(T_n\) acting on
\(4n\)-dimensional spaces \(Y_n\).  The construction is rooted in a
rank-two Nichols algebra and is strikingly explicit.  The colored
Links--Gould invariant, by contrast, is built from typical modules over a
quantum deformation of \(\mathfrak{sl}(2|1)\).  The two theories therefore
begin with rather different algebraic input.  The purpose of this paper is
to explain why they nevertheless produce the same knot invariant.

The connection is not obtained by a direct comparison of two large
matrices.  Instead, we expose a sequence of structural identifications:
\[
\begin{gathered}
\text{Nichols braid operator }T_n
\longleftrightarrow
\text{super Yetter--Drinfeld braiding}\\
\longleftrightarrow
\text{braiding of a completed paired double}\\
\xrightarrow{\ \text{root--center twist}\ }
U_\hbar^{\Pi_{\ao}}(\mathfrak{sl}(2|1))
\hot \mathcal Z_{\cen}.
\end{gathered}
\]
The \(4n\)-dimensional object splits at the last stage as
\[
 \mathcal Y_n\cong W_{n,\beta}^{\ao}\boxtimes L_{n,\beta},
\]
where \(W_{n,\beta}^{\ao}\) is a simple all-odd highest-weight module and
\(L_{n,\beta}\) is one-dimensional and central.  The central line
contributes a computable crossing scalar.  It disappears on balanced
normal long-knot diagrams and, more generally, under the standard writhe
normalization.

There are two points at which a braid-level comparison could fail to
produce a knot invariant.  First, the Hopf pairing used to form the double
must be well defined, nondegenerate in the relevant completion, and strong
enough to produce a genuine universal \(R\)-matrix.  We prove this by a
PBW normal form and explicit root-degree pairing matrices.  Second, a
conjugacy on \(Y_n\otimes Y_n\) does not by itself control cups, caps,
mixed orientations, partial transposes, or curls.  We therefore transport
the entire rigid tangle calculus and compare the transported duality maps
with the ordinary finite-dimensional ones.

\subsection{Related work and scope}

The original Links--Gould invariant was constructed from the
four-dimensional family of typical modules over
\(U_q(\mathfrak{gl}(2|1))\), using a tensor-state model for
\((1,1)\)-tangles \cite{DeWitKauffmanLinks}.  This is an instance of the
Reshetikhin--Turaev construction from a ribbon category
\cite{ReshetikhinTuraev}.  Because typical modules for Lie superalgebras
have vanishing ordinary quantum dimension, the natural closed-link theory
is more conveniently expressed using a renormalized invariant or modified
trace.  For \(\mathfrak{sl}(2|1)\), this point of view was developed by
Geer and Patureau--Mirand \cite{GeerPatureau}, and it underlies the modern
colored Links--Gould family.

The relevant algebraic background has several distinct layers.  The
classification and root theory of basic classical Lie superalgebras go
back to Kac \cite{Kac}, while changes of Borel through isotropic odd roots
belong to the generalized-root-system framework of Serganova
\cite{Serganova}.  Quantized enveloping superalgebras and their universal
\(R\)-matrices were constructed explicitly by Yamane \cite{Yamane}.
Etingof--Kazhdan quantization in the super setting and its relation to
Drinfeld--Jimbo presentations, including arbitrary Cartan matrices, were
studied by Geer \cite{GeerEK,GeerCartan}.  Geer's comparison with the
Kontsevich integral identifies the corresponding quantum-group and
weight-system invariants for Lie superalgebras of type A--G
\cite{GeerKontsevich}.  These general results supply essential structural
background, but they do not by themselves identify the particular
right--right completed double, module gauges, duality conventions, and
writhe normalization used below; those interfaces are treated explicitly
in this paper.

On the Nichols-algebra side, the spaces \(Y_n\) and their braid operators
were introduced by Garoufalidis and Kashaev in their construction of
multivariable knot polynomials from braided Hopf algebras with
automorphisms \cite{GK}.  Garoufalidis, Harper, Kashaev, Kohli, and Wagner
subsequently defined the colored Links--Gould sequence from
\(4n\)-dimensional typical modules, established its cabling and modified
trace framework, and formulated the equality with the \(V_n\)-sequence as
a conjecture; in that work the comparison was proved for \(n=2\)
\cite[Conjecture~1.4 and Theorem~1.5]{ColoredLG}.  Thus the contribution
claimed here is not the definition of either family.  It is the proposed
all-\(n\) structural bridge: the realization of the Nichols operator in a
completed paired double, its root--center separation, the identification
of the coloring object, and the transport of the full oriented rigid
tangle data needed to compare the normalized knot scalars.

We work over the formal parameter ring \(\mathbb C[[\hbar]]\), away from the
atypical values \(\beta=0,-1\).  The completion used below is the directed
root-degree completion determined by the explicit graded Hopf pairing; no
identification with an unrestricted continuous dual is asserted.  The
comparison concerns normalized oriented-knot invariants (equivalently,
endomorphisms of a normal \((1,1)\)-tangle followed by the stated framing
normalization).  It does not assert an equivalence of the entire module
categories of the two completed Hopf superalgebras.

\subsection{Main results}

Let \(\kk\) be a field of characteristic zero,
\[
 R=\kk[[\hbar]],\qquad \hK=\kk((\hbar)),\qquad Q=e^\hbar.
\]
The invariant \(LG_{\mathcal K}^{(n)}(q^\alpha,q)\) is normalized using the
\(4n\)-dimensional typical distinguished-Borel module
\(V_0(n-1,\alpha)\), with value \(1\) on the unknot.

\begin{theorem}[Main theorem]\label{thm:main}
For every \(n\geq 1\), every \(\beta\in\kk\setminus\{0,-1\}\), and every
oriented knot \(\mathcal K\),
\begin{equation}\label{eq:main}
 V_{n,\mathcal K}\!\left(Q^{2n\beta+n},Q^2\right)
 =
 LG_{\mathcal K}^{(n)}\!\left(Q^{-n\beta},Q^{-1}\right)
 =
 LG_{\overline{\mathcal K}}^{(n)}\!\left(Q^{n\beta},Q\right).
\end{equation}
Moreover, the endomorphism-valued long-knot invariant defining \(V_n\)
is a scalar multiple of the identity on \(\mathcal Y_n\).
\end{theorem}

The algebraic mechanism behind \cref{thm:main} is summarized by the next
two statements.

\begin{theorem}[Completed double and root--center factorization]
\label{thm:factorization}
The positive and negative Hopf superalgebras introduced in
\cref{sec:double} admit a continuous even Hopf pairing \(\eta\).  It is
perfect in each root degree, and its signed coevaluation series belongs
to the directed root-degree completion.  Consequently, the right--right
paired double \(\mathfrak D_{\mathrm{rr},\eta}\) has a completed universal
\(R\)-matrix \(\mathcal R_{\mathrm{rr}}\).

There is an explicit normalized Abelian Drinfeld twist
\(\mathcal F_{\rc}\) for which
\begin{equation}\label{eq:intro-factorization}
 \mathfrak D_{\mathrm{rr},\eta}^{\mathcal F_{\rc}}
 \cong
 \left(\hK\hot_R
 U_\hbar^{\Pi_{\ao}}(\mathfrak{sl}(2|1))\right)
 \hot\mathcal Z_{\cen},
\end{equation}
and the twisted universal \(R\)-matrix factors into a root part, a root
Cartan part, and a central Cartan part.
\end{theorem}

\begin{theorem}[Module and rigid-braiding transport]
\label{thm:transport-intro}
For every \(n\geq1\) and \(\beta\neq0,-1\), there is an isomorphism of
right modules
\[
 \mathcal Y_n\cong W_{n,\beta}^{\ao}\boxtimes L_{n,\beta}.
\]
Here \(W_{n,\beta}^{\ao}\) is the simple \(4n\)-dimensional all-odd
\(U_\hbar(\mathfrak{sl}(2|1))\)-module with highest weight
\[
 (n\beta,-n(\beta+1)).
\]
The factor \(L_{n,\beta}\) is one-dimensional and central.  The gauge
equivalence between \(T_n\) and the resulting braiding extends coherently
to dual objects, evaluation, coevaluation, all four oriented crossings,
and zero-framing normalization.
\end{theorem}

\subsection{Outline of the proof}

\Cref{sec:reduced} defines the reduced Hopf superalgebra and the spaces
\(\mathcal Y_n\).  In \cref{sec:yd} we construct a right--right
super Yetter--Drinfeld structure and give an explicit simultaneous gauge
for every braid generator.  The paired double is constructed in
\cref{sec:double}.  The main technical point there is the determinant
calculation in \cref{prop:perfect-pairing}, which removes any
nondegeneracy or convergence hypothesis from the subsequent arguments.

\Cref{sec:root-center} introduces the root and central Cartan directions.
The twist \(\mathcal F_{\rc}\) removes every mixed term in the coproduct
and in the Cartan part of the universal \(R\)-matrix.  In
\cref{sec:module} we identify \(\mathcal Y_n\) with a simple all-odd
module times a central line.  The proof is constructive: the top vector
is shown to be cyclic under the negative roots, and a degree-by-degree
calculation shows that it is the unique highest vector.

\Cref{sec:rigid} upgrades the braid equivalence to a rigid-tangle
equivalence.  The central line contributes
\(\lambda_{n,\beta}^{\cen}=Q^{2n^2\beta(\beta+1)}\) at every positive
crossing and its inverse at every negative crossing.  These factors cancel
on balanced normal diagrams and under writhe normalization.  Finally,
\cref{sec:LG} passes from the all-odd Borel to the distinguished Borel,
calibrates the right-module braiding direction, and proves
\cref{thm:main}.

\section{Conventions}\label{sec:conventions}

All tensor products over \(R\) are \(\hbar\)-adically completed; after
extension to \(\hK\), root-degree products are completed in the directed
sense specified in \cref{sec:completion}.  For an even element \(H\),
\[
 Q^H=e^{\hbar H}.
\]
If \(a,b,c,d\) are homogeneous elements of super vector spaces, then
\begin{equation}\label{eq:koszul}
 (a\ot b)(c\ot d)=(-1)^{|b||c|}ac\ot bd.
\end{equation}

Every module in this paper is a right module.  Thus, if
\(\rho(X)(v)=v\tri X\), then
\begin{equation}\label{eq:right-antihom}
 \rho(XY)=\rho(Y)\rho(X).
\end{equation}
Our universal \(R\)-matrix convention is
\begin{equation}\label{eq:R-convention}
 \mathcal R\Delta(a)=\Delta^{\mathrm{op}}(a)\mathcal R.
\end{equation}
It follows from \eqref{eq:right-antihom} that the canonical braiding on
right modules is
\begin{equation}\label{eq:right-braiding}
 c^{\mathcal R}_{V,W}
 =\tau^{\mathrm s}\rho_{V\ot W}(\mathcal R^{-1}),
\end{equation}
where \(\tau^{\mathrm s}\) is the super flip.  This inverse is responsible
for the mirror in \cref{eq:main}.

\section[The reduced Hopf superalgebra and Yn]
{The reduced Hopf superalgebra and \(\mathcal Y_n\)}
\label{sec:reduced}

\subsection{The reduced positive half}

Let \(\overline H_\hbar^{\mathrm s}\) be the topological Hopf
superalgebra generated by even elements \(H_1,H_2\) and odd elements
\(x_1,x_2\), subject to
\begin{equation}\label{eq:positive-relations}
 [H_1,H_2]=0,\qquad
 [H_i,x_j]=a_{ij}x_j,\qquad x_1^2=x_2^2=0,
 \qquad
 A=(a_{ij})=\begin{pmatrix}0&1\\1&0\end{pmatrix}.
\end{equation}
No additional quadratic relation is imposed between \(x_1x_2\) and
\(x_2x_1\).  Put \(k_i=Q^{H_i}\).  The Hopf structure is
\begin{equation}\label{eq:positive-Hopf}
\begin{aligned}
 \Delta(H_i)&=H_i\ot1+1\ot H_i,
 &\varepsilon(H_i)&=0,
 &S(H_i)&=-H_i,\\
 \Delta(x_i)&=x_i\ot1+k_i\ot x_i,
 &\varepsilon(x_i)&=0,
 &S(x_i)&=-k_i^{-1}x_i.
\end{aligned}
\end{equation}
\subsection[The 4n-dimensional space]{The \(4n\)-dimensional space}

Let \(B_{\hK}\) be the superalgebra generated by \(x_1,x_2\) with only
\(x_1^2=x_2^2=0\).  A homogeneous word of bidegree \((a,b)\) has parity
\(\overline{a+b}\).  The full diagonal braiding is determined by
\begin{equation}\label{eq:chi}
 \chi((a,b),x_1)=(-1)^{a+b}Q^b,\qquad
 \chi((a,b),x_2)=(-1)^{a+b}Q^a.
\end{equation}
Fix \(n\geq1\) and set
\begin{equation}\label{eq:parameters}
\begin{aligned}
 z&=Q^{2\beta+1},&
 t&=z^n=Q^{2n\beta+n},\\
 t_{1,n}&=Q^{-2n(\beta+1)},&
 t_{2,n}&=Q^{2n\beta}.
\end{aligned}
\end{equation}
Thus \(t_{1,n}t_{2,n}Q^{2n}=1\).  Define
\[
 v_n=(x_1x_2)^n+t(x_2x_1)^n.
\]
For \(\beta\neq0,-1\), let \(\mathcal Y_n\) be the span of
\begin{equation}\label{eq:Yn-basis}
\begin{split}
 \{1,v_n\}
 &\cup\{(x_1x_2)^k,(x_2x_1)^k:1\leq k<n\}\\
 &\cup\{(x_1x_2)^kx_1,(x_2x_1)^kx_2:0\leq k<n\}.
\end{split}
\end{equation}
This displayed set is a basis, so \(\dim_{\hK}\mathcal Y_n=4n\).

The parameter and convention dictionary identifying \(\mathcal Y_n\) with
the right Yetter--Drinfeld module \(Y_n\) of
\cite[Proposition~7.4]{GK} is given in
\cref{app:GK-conventions}.  Under this identification, \(T_n\) denotes
precisely the braid-form \(R\)-matrix defined by
\cite[Theorem~3.6 and Definition~7.5]{GK}, after extension of scalars to
\(\hK\).

\section[A super Yetter--Drinfeld realization of Tn]
{A super Yetter--Drinfeld realization of \(T_n\)}
\label{sec:yd}

\subsection{Action and coaction}

Introduce the group-like element
\begin{equation}\label{eq:Un}
 U_n=Q^{-2n\beta H_1+2n(\beta+1)H_2}.
\end{equation}
For homogeneous \(y\in\mathcal Y_n\) of bidegree \((a,b)\), define
\begin{equation}\label{eq:Yn-action}
\begin{aligned}
 y\tri H_1&=-by,& y\tri H_2&=-ay,\\
 y\tri x_i&=yx_i-t_{i,n}\chi(y,x_i)x_i y.
\end{aligned}
\end{equation}
The boundary identities
\[
 v_n\tri x_1=v_n\tri x_2=0,\qquad
 ((x_1x_2)^{n-1}x_1)\tri x_2=v_n,\qquad
 ((x_2x_1)^{n-1}x_2)\tri x_1=t^{-1}v_n
\]
show directly that the action preserves \(\mathcal Y_n\).

The coaction uses the even torus bicharacter
\begin{equation}\label{eq:r}
 r((c,d),(a,b))=Q^{cb+da}.
\end{equation}
Write the braided coproduct in \(B_{\hK}\) as
\[
 \Delta_B(y)=\sum y^{(1)}\ot y^{(2)},\qquad
 \deg y^{(1)}=p,\quad \deg y^{(2)}=q.
\]
Then
\begin{equation}\label{eq:Yn-coaction}
 \delta_n^{\mathrm s}(y)
 =
 \sum r(q,p)^{-1}y^{(1)}\ot
 U_nk_1^{-a}k_2^{-b}y^{(2)}.
\end{equation}
It is important that \eqref{eq:Yn-coaction} contains \(r^{-1}\), not
\(\chi^{-1}\).  The missing sign is already supplied by the Koszul
interchange in the super tensor category.

\begin{proposition}\label{prop:YD}
Equations \eqref{eq:Yn-action} and \eqref{eq:Yn-coaction} make
\(\mathcal Y_n\) a right--right super Yetter--Drinfeld module over
\(\overline H_{\hK}^{\mathrm s}\).
\end{proposition}

\begin{proof}
For \(y\) of bidegree \((a,b)\), set
\[
 \iota_n(y)=U_nk_1^{-a}k_2^{-b}y\in\overline H_{\hK}^{\mathrm s}.
\]
The action \eqref{eq:Yn-action} is the pullback along \(\iota_n\) of the
right super-adjoint action
\[
 A\tri_{\mathrm{ad}}h
 =\sum(-1)^{|A||h_{(1)}|}S(h_{(1)})Ah_{(2)}.
\]
Moreover,
\[
 \Delta(\iota_n(y))=(\iota_n\ot\id)\delta_n^{\mathrm s}(y).
\]
Thus \(\iota_n(\mathcal Y_n)\) is stable under the right adjoint action
and is a right coideal.  The standard adjoint/coideal calculation, with
the Koszul rule \eqref{eq:koszul}, gives the right--right super
Yetter--Drinfeld identity.  Counitality and coassociativity follow by
restricting those of \(\Delta\).
\end{proof}

\subsection{The simultaneous braid gauge}

For a right--right super Yetter--Drinfeld module, the checked braiding is
\begin{equation}\label{eq:YD-braiding}
 c_n^{\mathrm s}(a\ot b)
 =\sum(-1)^{|a||b_{(0)}|}
 b_{(0)}\ot(a\tri b_{(1)}).
\end{equation}
Let \(\Psi_n(a)=a\tri U_n\), and for homogeneous \(a,b\) define
\[
 J_2^{\mathrm s}(a\ot b)=r(\deg b,\deg a)a\ot b,\qquad
 G_2^{\mathrm s}=J_2^{\mathrm s}(\Psi_n\ot1).
\]

\begin{proposition}\label{prop:braid-gauge}
Under the identification of \cref{app:GK-conventions}, the
Garoufalidis--Kashaev braid operator satisfies
\begin{equation}\label{eq:T-conjugacy}
 T_n=G_2^{\mathrm s}c_n^{\mathrm s}(G_2^{\mathrm s})^{-1}.
\end{equation}
More generally, on \(\mathcal Y_n^{\ot m}\) set
\[
\begin{aligned}
 D_m^{\mathrm s}(y_1\ot\cdots\ot y_m)
 &=
 \left(\prod_{i<j}r(\deg y_j,\deg y_i)\right)
 y_1\ot\cdots\ot y_m,\\
 G_m^{\mathrm s}
 &=D_m^{\mathrm s}
 (\Psi_n^{m-1}\ot\Psi_n^{m-2}\ot\cdots\ot\Psi_n\ot1).
\end{aligned}
\]
Then every adjacent generator is simultaneously conjugated:
\begin{equation}\label{eq:all-braid-gauge}
 (T_n)_i=G_m^{\mathrm s}(c_n^{\mathrm s})_i
 (G_m^{\mathrm s})^{-1},\qquad 1\leq i<m.
\end{equation}
\end{proposition}

\begin{proof}
Here \(a\tri_B h=\lambda_R(a\ot h)\) denotes the
Garoufalidis--Kashaev twisted right-adjoint action; see
\cref{app:GK-action}.  Let \(\deg a=p\), and write
\(\Delta_B(b)=\sum b^{(1)}\ot b^{(2)}\) with degrees
\(\rho,\sigma\).  Expanding \eqref{eq:YD-braiding} gives
\[
 c_n^{\mathrm s}(a\ot b)
 =\sum(-1)^{\bar p\bar\rho}
 r(\sigma,\rho)^{-1}r(\rho+\sigma,p)\,
 b^{(1)}\ot(\Psi_n(a)\tri_B b^{(2)}).
\]
Conjugation by \(G_2^{\mathrm s}\) changes the scalar to
\[
 (-1)^{\bar p\bar\rho}
 r(\sigma,\rho)^{-1}r(p+\sigma,\rho)
 =(-1)^{\bar p\bar\rho}r(p,\rho)
 =\chi(p,\rho),
\]
which is exactly the coefficient in the defining formula for \(T_n\).
This proves \eqref{eq:T-conjugacy}.  In higher tensor powers, factors
involving spectator strands cancel pairwise; the remaining two-strand
calculation is the one just performed.
\end{proof}

\begin{remark}\label{rem:braids-not-knots}
\Cref{prop:braid-gauge} identifies braid group representations.  It does
not yet identify knot invariants, because a closure also uses duals,
evaluation, coevaluation, mixed crossings, partial transposes, and
framing normalization.  These data are treated in \cref{sec:rigid}.
\end{remark}

\section{The completed right--right paired double}\label{sec:double}

\subsection{The negative half and the Hopf pairing}

Let \(\mathcal H^+=\overline H_{\hK}^{\mathrm s}\).  Define
\(\mathcal H^-\) using even generators \(\widehat H_1,\widehat H_2\),
odd generators \(F_1,F_2\), and \(L_i=Q^{\widehat H_i}\), with
\[
 [\widehat H_1,\widehat H_2]=0,\qquad
 [\widehat H_i,F_j]=a_{ji}F_j,\qquad F_i^2=0,
\]
and
\begin{equation}\label{eq:negative-Hopf}
\begin{aligned}
 \Delta(\widehat H_i)&=\widehat H_i\ot1+1\ot\widehat H_i,&
 S(\widehat H_i)&=-\widehat H_i,\\
 \Delta(F_i)&=F_i\ot L_i+1\ot F_i,&
 S(F_i)&=-F_iL_i^{-1}.
\end{aligned}
\end{equation}
The even Hopf pairing
\(\eta:\mathcal H^-\hot\mathcal H^+\to\hK\) is determined by
\begin{equation}\label{eq:eta-generators}
 \eta(\widehat H_i,H_j)=-\frac{a_{ji}}{\hbar},
 \qquad
 \eta(F_i,x_j)=\delta_{ij},
\end{equation}
all mixed Cartan--root pairings being zero, and by the right-dual order
\begin{equation}\label{eq:eta-extension}
\eta(fg,h)=\sum\eta(g,h_{(1)})\eta(f,h_{(2)}),\qquad
\eta(f,hk)=\sum\eta(f_{(1)},h)\eta(f_{(2)},k),
\end{equation}
with the Koszul sign understood in the second formula.

\subsection{PBW normal forms and perfectness}\label{sec:completion}

For \(m\geq1\), let \(X_{i,m}\) be the alternating word of length \(m\)
in \(x_1,x_2\) beginning with \(x_i\), and let \(F_{i,m}\) be its
negative analogue.  Since adjacent squares are the only monomial
relations, the root pieces are
\[
 B_m^+=\hK X_{1,m}\oplus\hK X_{2,m},\qquad
 B_m^-=\hK F_{1,m}\oplus\hK F_{2,m}.
\]
Moving Cartan generators to the left gives topological PBW decompositions
\[
 \mathcal H^\pm
 \cong\widehat{\bigoplus}_{m\geq0}\mathcal C^\pm B_m^\pm.
\]
The Diamond-lemma ambiguities are harmless: an overlap of square
relations still contains an adjacent square, while a Cartan--root rewrite
only moves a Cartan element past a root vector of fixed weight.

\begin{lemma}\label{lem:pairing-descends}
The formulas \eqref{eq:eta-generators}--\eqref{eq:eta-extension} descend
uniquely to a continuous even Hopf pairing on
\(\mathcal H^-\hot\mathcal H^+\), and
\(\eta(B_m^-,B_l^+)=0\) for \(m\neq l\).
\end{lemma}

\begin{proof}
In the free Hopf superalgebras, the odd--odd terms in the coproduct cancel
by the Koszul rule, so the square relations on either side annihilate the
relation ideal on the other.  Primitive Cartan elements preserve Cartan
commutativity, while
\(\eta(L_i,k_j)=Q^{-a_{ji}}\) preserves the fixed-weight
Cartan--root relations.  Hence both relation ideals lie in the radicals
of the free pairing, and the pairing descends.  Root-degree orthogonality
follows from the skew-primitive coproducts.
\end{proof}

Put \(P_r(Q)=\prod_{s=1}^r(1-Q^{-2s})\), with \(P_0(Q)=1\), and let
\[
 M_m(i,j)=\eta(F_{i,m},X_{j,m}).
\]

\begin{proposition}[Root-degree perfectness]\label{prop:perfect-pairing}
For \(r\geq0\) and \(r\geq1\), respectively,
\begin{equation}\label{eq:pairing-matrices}
\begin{aligned}
 M_{2r+1}&=(-1)^rP_r(Q)I_2,\\
 M_{2r}&=(-1)^rP_{r-1}(Q)
 \begin{pmatrix}1&-Q^{-r}\\-Q^{-r}&1\end{pmatrix}.
\end{aligned}
\end{equation}
Consequently,
\[
 \det M_{2r+1}=P_r(Q)^2,\qquad
 \det M_{2r}=P_{r-1}(Q)P_r(Q),
\]
and \(\eta|_{B_m^-\ot B_m^+}\) is perfect for every \(m\).
\end{proposition}

\begin{proof}
Apply \eqref{eq:eta-extension} recursively to the last letter of each
alternating word.  The two possible starting letters give a two-state
recurrence.  At odd length its off-diagonal entries cancel, while at even
length the two possible terminal pairings differ by \(Q^{-r}\).  Starting
from \(M_1=I_2\) and
\[
 M_2=\begin{pmatrix}-1&Q^{-1}\\Q^{-1}&-1\end{pmatrix}
\]
gives \eqref{eq:pairing-matrices} by induction.  Taking determinants and
using \(P_r=P_{r-1}(1-Q^{-2r})\) proves the claim.  Every \(P_r(Q)\) is
nonzero in \(\hK=\kk((\hbar))\).
\end{proof}

We use the directed root-degree completion
\begin{equation}\label{eq:directed-completion}
 \widehat{\mathcal T}_{+,-}
 =
 \prod_{m\geq0}
 (\mathcal C^+B_m^+)\hot_R(\mathcal C^-B_m^-).
\end{equation}
Only the common root degree is completed; Cartan coefficients remain
\(\hbar\)-adically completed.  Root degrees add under multiplication, so
every fixed degree of a product contains only finitely many summands.

\begin{theorem}[Completed universal \(R\)-matrix]\label{thm:universal-R}
The right--right paired double
\[
 \mathfrak D_{\mathrm{rr},\eta}=\mathcal H^-\bowtie_\eta\mathcal H^+
\]
is well defined.  The sum of the signed right-dual coevaluation tensors
in all root degrees belongs to \eqref{eq:directed-completion}; denote it
by \(\mathcal B_{\mathrm{rr}}^{\mathrm{can}}\).  Then
\[
 \mathcal R_{\mathrm{rr}}
 =(\mathcal B_{\mathrm{rr}}^{\mathrm{can}})^{-1}
\]
exists in the same directed completion and satisfies
\begin{align}
 \mathcal R_{\mathrm{rr}}\Delta(d)
 &=\Delta^{\mathrm{op}}(d)\mathcal R_{\mathrm{rr}},
 \label{eq:R-intertwines}\\
 (\Delta\ot\id)(\mathcal R_{\mathrm{rr}})
 &=\mathcal R_{13}\mathcal R_{23},&
 (\id\ot\Delta)(\mathcal R_{\mathrm{rr}})
 &=\mathcal R_{13}\mathcal R_{12}.
\label{eq:R-hexagons}
\end{align}
\end{theorem}

\begin{proof}
\Cref{prop:perfect-pairing} supplies a dual basis in each root degree.
The super right-dual coevaluation contributes the additional sign
\((-1)^m\) in degree \(m\).  The resulting series belongs to
\eqref{eq:directed-completion}, and its inverse is defined recursively
degree by degree.  Dualizing multiplication and coproduct through
\eqref{eq:eta-extension} gives
\[
\begin{aligned}
 \Delta(d)\mathcal B_{\mathrm{rr}}^{\mathrm{can}}
 &=\mathcal B_{\mathrm{rr}}^{\mathrm{can}}\Delta^{\mathrm{op}}(d),\\
 (\Delta\ot\id)(\mathcal B_{\mathrm{rr}}^{\mathrm{can}})
 &=\mathcal B_{23}^{\mathrm{can}}\mathcal B_{13}^{\mathrm{can}},\\
 (\id\ot\Delta)(\mathcal B_{\mathrm{rr}}^{\mathrm{can}})
 &=\mathcal B_{12}^{\mathrm{can}}\mathcal B_{13}^{\mathrm{can}}.
\end{aligned}
\]
Inverting these identities proves \eqref{eq:R-intertwines} and
\eqref{eq:R-hexagons}.  No identification of either half with the full
continuous dual of the other is required.
\end{proof}

The cross relations in the double are
\begin{equation}\label{eq:double-cross}
\begin{aligned}
 [H_i,\widehat H_j]&=0,&
 [H_i,F_j]&=-a_{ij}F_j,\\
 [\widehat H_i,x_j]&=-a_{ji}x_j,&
 x_iF_j+F_jx_i&=\delta_{ij}(k_i-L_i).
\end{aligned}
\end{equation}
The double action induced from the Yetter--Drinfeld structure agrees,
through \eqref{eq:right-braiding}, with the canonical braiding
\(c_n^{\mathrm s}\).

\section{The root--center twist}\label{sec:root-center}

\subsection{Root and central variables}

Define
\begin{equation}\label{eq:root-center-Cartan}
 \mathsf h_i=\frac{H_i-\widehat H_i}{2},\qquad
 \mathsf c_i=H_i+\widehat H_i,\qquad
 \kappa_i=Q^{\mathsf h_i},\qquad
 \zeta_i=Q^{\mathsf c_i}=k_iL_i.
\end{equation}
The \(\mathsf c_i\) are central.  Normalize the odd root vectors by
\begin{equation}\label{eq:root-generators}
 E_i=\zeta_i^{-1/4}x_i,\qquad
 G_i=\zeta_i^{-1/4}F_i.
\end{equation}
This separates the multiplication but not yet the coproduct.

\subsection{The Abelian twist}

Because \(A^{-1}=A\), define
\begin{equation}\label{eq:Frc}
\mathcal F_{\rc}
=
\exp\!\left[
\frac{\hbar}{4}
\sum_{i,j=1}^2(A^{-1})_{ji}
\bigl(\mathsf h_i\ot\mathsf c_j-\mathsf c_j\ot\mathsf h_i\bigr)
\right].
\end{equation}
All entries in the exponent are commuting even primitive elements.
Therefore \(\mathcal F_{\rc}\) is a normalized Drinfeld twist.  We use
\[
 \Delta_{\rc}(a)=\mathcal F_{\rc}\Delta(a)\mathcal F_{\rc}^{-1}.
\]

\begin{theorem}[Root--center separation]\label{thm:root-center}
After twisting, the root and central Hopf structures separate:
\begin{equation}\label{eq:Hopf-factor}
 \mathfrak D_{\mathrm{rr},\eta}^{\mathcal F_{\rc}}
 \cong\mathcal D_{\rootpart}\hot\mathcal Z_{\cen}.
\end{equation}
Moreover,
\[
 \mathcal D_{\rootpart}
 \cong
 \hK\hot_R U_\hbar^{\Pi_{\ao}}(\mathfrak{sl}(2|1)),
\]
under
\[
 h_i\mapsto\mathsf h_i,\qquad
 K_i\mapsto\kappa_i,\qquad
 e_i\mapsto E_i,\qquad
 f_i\mapsto(Q-Q^{-1})^{-1}G_i.
\]
\end{theorem}

For clarity, the complete root presentation used in this theorem is
\begin{equation}\label{eq:all-odd-root-presentation}
\begin{gathered}
 [\mathsf h_i,\mathsf h_j]=0,
 \qquad [\mathsf h_i,E_j]=a_{ij}E_j,
 \qquad [\mathsf h_i,G_j]=-a_{ij}G_j,\\
 E_1^2=E_2^2=G_1^2=G_2^2=0,
 \qquad
 E_iG_j+G_jE_i
 =\delta_{ij}(\kappa_i-\kappa_i^{-1}).
\end{gathered}
\end{equation}
No additional quadratic relation is imposed between \(E_1E_2\) and
\(E_2E_1\), or between \(G_1G_2\) and \(G_2G_1\).  Every central
generator commutes with every generator in
\eqref{eq:all-odd-root-presentation}.

\begin{proof}
Conjugating the four skew-primitive terms by
\(\mathcal F_{\rc}\) gives
\[
\begin{aligned}
 \Delta_{\rc}(E_i)&=E_i\ot1+\kappa_i\ot E_i,\\
 \Delta_{\rc}(G_i)&=G_i\ot\kappa_i^{-1}+1\ot G_i.
\end{aligned}
\]
The \(\mathsf h_i,\mathsf c_i\) remain primitive, and \(\kappa_i,\zeta_i\)
remain group-like.  Thus no central generator occurs in the coproduct of
a root generator.  Conversely, the \(\mathsf c_i,\zeta_i^{\pm1/4}\)
generate a central commutative Hopf algebra.  The multiplication relations
from \eqref{eq:double-cross} become the all-odd
\(\mathfrak{sl}(2|1)\) relations, with both simple roots isotropic and
Cartan matrix \(A\).  The PBW normal forms on both sides show that the
resulting surjective Hopf map is injective.  This proves
\eqref{eq:Hopf-factor}.
\end{proof}

\subsection[Factorization of the universal R-matrix]
{Factorization of the universal \(R\)-matrix}

Since the exponent of \(\mathcal F_{\rc}\) is antisymmetric,
\(\mathcal F_{\rc,21}=\mathcal F_{\rc}^{-1}\).  Hence
\[
 \mathcal R_{\rc}
 =\mathcal F_{\rc,21}\mathcal R_{\mathrm{rr}}
  \mathcal F_{\rc}^{-1}
 =\mathcal F_{\rc}^{-1}\mathcal R_{\mathrm{rr}}
  \mathcal F_{\rc}^{-1}.
\]
Let \(\mathcal R_{\rootpart}^{\rc}\) denote the inverse of the signed
canonical root tensor after replacing alternating \(x\)- and \(F\)-words
by the corresponding \(E\)- and \(G\)-words.  Define
\[
\begin{aligned}
 \mathcal R_{\mathrm{Cart}}^{\rootpart}
 &=
 \exp\!\left[-\hbar(
 \mathsf h_1\ot\mathsf h_2+\mathsf h_2\ot\mathsf h_1)\right],\\
 \mathcal R_{\mathrm{Cart}}^{\cen}
 &=
 \exp\!\left[\frac{\hbar}{4}(
 \mathsf c_1\ot\mathsf c_2+\mathsf c_2\ot\mathsf c_1)\right].
\end{aligned}
\]

\begin{proposition}\label{prop:R-factor}
In the directed completion,
\begin{equation}\label{eq:R-factor}
 \mathcal R_{\rc}
 =
 \mathcal R_{\rootpart}^{\rc}
 \mathcal R_{\mathrm{Cart}}^{\rootpart}
 \mathcal R_{\mathrm{Cart}}^{\cen}.
\end{equation}
This is a completed universal \(R\)-matrix for
\(\Delta_{\rc}\).
\end{proposition}

\begin{proof}
Substitution of
\(H_i=\mathsf h_i+\frac12\mathsf c_i\) and
\(\widehat H_i=-\mathsf h_i+\frac12\mathsf c_i\) into the untwisted
Cartan exponential yields pure-root, mixed, and pure-central terms.
The two factors \(\mathcal F_{\rc}^{-1}\) cancel the mixed terms exactly.
They simultaneously turn every root tensor
\(X_{i,m}\ot F_{j,m}\) into \(E_{i,m}\ot G_{j,m}\).
This proves \eqref{eq:R-factor}.  The universal \(R\)-identities follow
either by twisting \cref{thm:universal-R} or by direct substitution into
\eqref{eq:R-intertwines}--\eqref{eq:R-hexagons}.
\end{proof}

\section[The 4n-dimensional all-odd module]
{The \(4n\)-dimensional all-odd module}\label{sec:module}

On a homogeneous \(y\in\mathcal Y_n\) of degree \((a,b)\), the root and
central Cartan elements act by
\begin{equation}\label{eq:root-center-actions}
\begin{aligned}
 y\tri\mathsf h_1&=(n(\beta+1)-b)y,&
 y\tri\mathsf h_2&=-(n\beta+a)y,\\
 y\tri\mathsf c_1&=-2n(\beta+1)y,&
 y\tri\mathsf c_2&=2n\beta y.
\end{aligned}
\end{equation}
Let \(W_{n,\beta}^{\ao}\) be the all-odd right highest-weight module with
even highest vector \(w_{n,\beta}\) satisfying
\begin{equation}\label{eq:highest-weight}
 w_{n,\beta}\tri E_1=w_{n,\beta}\tri E_2=0,\qquad
 w_{n,\beta}\tri\mathsf h_1=n\beta w_{n,\beta},\qquad
 w_{n,\beta}\tri\mathsf h_2=-n(\beta+1)w_{n,\beta}.
\end{equation}
Let \(L_{n,\beta}\) be the one-dimensional central module with characters
\[
 \mathsf c_1\mapsto-2n(\beta+1),\quad
 \mathsf c_2\mapsto2n\beta,\quad
 \zeta_1\mapsto t_{1,n},\quad
 \zeta_2\mapsto t_{2,n}.
\]

\begin{theorem}[Object identification]\label{thm:object}
For \(\beta\neq0,-1\),
\begin{equation}\label{eq:object}
 \mathcal Y_n
 \cong W_{n,\beta}^{\ao}\boxtimes L_{n,\beta}
\end{equation}
as a right module over the twisted paired double.  The root factor
\(W_{n,\beta}^{\ao}\) is simple and has dimension \(4n\).
\end{theorem}

\begin{proof}
Write
\[
 A_k=(x_1x_2)^k,\quad B_k=(x_2x_1)^k,\quad
 C_k=(x_1x_2)^kx_1,\quad D_k=(x_2x_1)^kx_2.
\]
The vector \(v_n=A_n+tB_n\) is even, is killed by \(E_1,E_2\), and has
the weight in \eqref{eq:highest-weight}.

The negative root actions differ by nonzero degree scalars from braided
derivatives \(\partial_1,\partial_2\).  They obey
\[
\begin{array}{ll}
\partial_1(A_k)=-Q^kD_{k-1},&
\partial_2(A_k)=C_{k-1},\\
\partial_1(B_k)=D_{k-1},&
\partial_2(B_k)=-Q^kC_{k-1},
\end{array}
\]
and
\[
 \partial_1(v_n)=(t-Q^n)D_{n-1},\qquad
 \partial_2(v_n)=(1-tQ^n)C_{n-1}.
\]
Both coefficients are nonzero precisely under
\(\beta\neq0,-1\).  Furthermore,
\[
 \partial_1(C_k)=A_k+Q^kB_k,\qquad
 \partial_2(D_k)=Q^kA_k+B_k.
\]
The coefficient matrix has determinant \(1-Q^{2k}\neq0\).
Descending induction therefore produces every basis vector in
\eqref{eq:Yn-basis} from \(v_n\); hence \(v_n\) is cyclic.

It remains to show that \(v_n\) is the unique highest vector.  On
\(\operatorname{span}\{A_k,B_k\}\), \(1\leq k<n\), the common kernel of
\(E_1,E_2\) is controlled by
\[
 \begin{pmatrix}
 1&-t_{1,n}Q^k\\
 -t_{2,n}Q^k&1
 \end{pmatrix},
\]
whose determinant is \(1-Q^{2(k-n)}\neq0\).  The corresponding odd-degree
formulas show the same for \(\operatorname{span}\{C_k,D_k\}\), while the
top common kernel is exactly \(\hK v_n\).  Any nonzero submodule contains
a vector of maximal degree, hence a highest vector, and therefore
contains \(v_n\).  Cyclicity then forces the submodule to be all of
\(\mathcal Y_n\).  This proves simplicity and identifies the root factor.
The central action in \eqref{eq:root-center-actions} gives the tensor
factor \(L_{n,\beta}\), proving \eqref{eq:object}.
\end{proof}

\section{Rigid transport and normal long knots}\label{sec:rigid}

\subsection{The central crossing scalar}

On \(\mathcal Y_n\ot\mathcal Y_n\), the inverse of the central Cartan
factor in \eqref{eq:R-factor} acts by
\begin{equation}\label{eq:lambda}
 \lambda_{n,\beta}^{\cen}
 =Q^{2n^2\beta(\beta+1)}.
\end{equation}
Let
\[
 W=W_{n,\beta}^{\ao},\qquad
 R'_{n,\beta}=c^{\ao}_{W,W}
\]
be the all-odd root braiding obtained from the inverse of the root
universal \(R\)-matrix in the right-module convention.

Combining the original Yetter--Drinfeld gauge with the tensorator of
\(\mathcal F_{\rc}\) and the object isomorphism \eqref{eq:object} gives a
degree-diagonal isomorphism \(\widehat{\mathsf G}_{n,m}^{\rc}\) such that,
for every \(m\)-strand braid \(b\),
\begin{equation}\label{eq:braid-root-comparison}
\Phi_n^{[m]}\rho_{T_n}(b)(\Phi_n^{[m]})^{-1}
=
(\lambda_{n,\beta}^{\cen})^{\Writhe(b)}
\widehat{\mathsf G}_{n,m}^{\rc}
\rho_{\ao}(b)
(\widehat{\mathsf G}_{n,m}^{\rc})^{-1}.
\end{equation}

\subsection{Fixed-color oriented words and duals}
\label{sec:oriented-fixed-color}

A two-strand conjugacy is not enough for \cref{thm:main}.  Even when the
color \(n\) is fixed, an oriented tangle contains both \(\mathcal Y_n\)
and its left categorical dual.  We therefore make the fixed-color
oriented transport explicit.

Since \(\mathcal Y_n\) is finite-dimensional over \(\hK\), put
\[
 \mathcal Y_n^\vee=\operatorname{Hom}_{\hK}(\mathcal Y_n,\hK).
\]
For homogeneous \(f\in\mathcal Y_n^\vee\), \(h\) in the paired double,
and \(y\in\mathcal Y_n\), its right action is
\begin{equation}\label{eq:dual-right-action}
 \langle f\tri h,y\rangle
 =
 (-1)^{|h|(|f|+1)}
 \langle f,y\tri S^{-1}(h)\rangle.
\end{equation}
This is the unique action for which evaluation is a right-module
morphism.  If \(\{e_a\}\) is a homogeneous basis and \(\{e^a\}\) its
homogeneous dual basis, we use
\begin{equation}\label{eq:dual-degree-cups}
 \deg e^a=-\deg e_a,\qquad
 \ev_n(f\ot y)=f(y),\qquad
 \coev_n(1)=\sum_a e_a\ot e^a.
\end{equation}
The sign in \eqref{eq:dual-right-action} and the use of \(S^{-1}\) are
forced by the left-dual convention.  With them, \(\ev_n\) and
\(\coev_n\) are module morphisms and satisfy the two super snake
identities.

Introduce the two oriented colors
\begin{equation}\label{eq:oriented-objects-charges}
 X_{n^+}=\mathcal Y_n,\qquad
 X_{n^-}=\mathcal Y_n^\vee,\qquad
 U_{n^+}=U_n,\qquad U_{n^-}=U_n^{-1}.
\end{equation}
The inverse charge on \(X_{n^-}\) is not an extra normalization: it
follows from the antipode in \eqref{eq:dual-right-action}.

Let
\[
 \boldsymbol\alpha=(\alpha_1,\ldots,\alpha_r),
 \qquad
 \alpha_i\in\{n^+,n^-\},\qquad
 X_{\boldsymbol\alpha}
 =X_{\alpha_1}\ot\cdots\ot X_{\alpha_r}.
\]
For homogeneous \(x_i\in X_{\alpha_i}\) of bidegree \(p_i\), with the
negative degree convention in \eqref{eq:dual-degree-cups}, define the
oriented staircase gauge
\begin{equation}\label{eq:oriented-staircase}
\begin{aligned}
G_{\boldsymbol\alpha}^{\mathrm{or}}
(x_1\ot\cdots\ot x_r)
={}&
\left(\prod_{1\leq i<j\leq r}r(p_j,p_i)\right)\\
&{}\cdot
\bigotimes_{i=1}^r
\left(x_i\tri\prod_{j=i+1}^rU_{\alpha_j}\right).
\end{aligned}
\end{equation}
All charges lie in the same commutative torus, so the products in
\eqref{eq:oriented-staircase} are unambiguous.  If every sign is positive,
this reduces to \(G_r^{\mathrm s}\) from \cref{prop:braid-gauge}.

For \(\alpha,\gamma\in\{n^+,n^-\}\), let
\[
 c_{\alpha,\gamma}^{\mathrm s}:
 X_\alpha\ot X_\gamma\longrightarrow X_\gamma\ot X_\alpha
\]
be the canonical braiding of the right-module category.  Define all four
oriented \(T\)-crossings at once by
\begin{equation}\label{eq:oriented-T-crossing}
 T_{\alpha,\gamma}
 =
 G_{(\gamma,\alpha)}^{\mathrm{or}}\,
 c_{\alpha,\gamma}^{\mathrm s}\,
 (G_{(\alpha,\gamma)}^{\mathrm{or}})^{-1}.
\end{equation}
Thus \eqref{eq:oriented-T-crossing} includes
\(T_{n^+,n^+}\), \(T_{n^-,n^+}\),
\(T_{n^+,n^-}\), and \(T_{n^-,n^-}\), and
\(T_{n^+,n^+}=T_n\).

\begin{proposition}[Oriented braid-groupoid transport]
\label{prop:oriented-transport}
Let \(s_i\boldsymbol\alpha\) be obtained by exchanging the \(i\)-th and
\((i+1)\)-st entries.  Then
\begin{equation}\label{eq:oriented-local-transport}
 (T_{\boldsymbol\alpha})_i
 =
 G_{s_i\boldsymbol\alpha}^{\mathrm{or}}\,
 (c_{\boldsymbol\alpha}^{\mathrm s})_i\,
 (G_{\boldsymbol\alpha}^{\mathrm{or}})^{-1}.
\end{equation}
Consequently, the four operators in
\eqref{eq:oriented-T-crossing} satisfy every fixed-color oriented braid
groupoid relation.
\end{proposition}

\begin{proof}
The gauge \eqref{eq:oriented-staircase} is a product of factors indexed
by ordered pairs of strands.  In the conjugation of an adjacent
crossing, every factor involving a spectator strand occurs once on each
side and cancels.  The remaining factor is precisely the two-strand
formula \eqref{eq:oriented-T-crossing}.  Conjugating the ordinary braid
relations by the boundary gauges proves the braid-groupoid relations.
\end{proof}

Set \(G_\varnothing^{\mathrm{or}}=\id_{\hK}\).  The transported cups and
caps are
\begin{equation}\label{eq:transported-cups}
\begin{aligned}
 \ev_n^T&=\ev_n(G_{(n^-,n^+)}^{\mathrm{or}})^{-1},\\
 \coev_n^T&=G_{(n^+,n^-)}^{\mathrm{or}}\coev_n.
\end{aligned}
\end{equation}
More generally, for
\(F:X_{\boldsymbol\alpha}\to X_{\boldsymbol\gamma}\), set
\begin{equation}\label{eq:boundary-gauge-transport}
 F^T
 =
 G_{\boldsymbol\gamma}^{\mathrm{or}}\,
 F\,
 (G_{\boldsymbol\alpha}^{\mathrm{or}})^{-1}.
\end{equation}
Boundary gauges cancel under composition.  Tensor products require the
cross tensorator
\begin{equation}\label{eq:cross-tensorator}
 \mathcal J_{\boldsymbol\alpha,\boldsymbol\gamma}^{T}
 =
 G_{\boldsymbol\alpha\boldsymbol\gamma}^{\mathrm{or}}
 \left(
 G_{\boldsymbol\alpha}^{\mathrm{or}}
 \ot G_{\boldsymbol\gamma}^{\mathrm{or}}
 \right)^{-1},
\end{equation}
where juxtaposition denotes concatenation of words.  If
\(F:X_{\boldsymbol\alpha}\to X_{\boldsymbol\alpha'}\) and
\(H:X_{\boldsymbol\gamma}\to X_{\boldsymbol\gamma'}\), define
\begin{equation}\label{eq:transported-tensor-product}
 F^T\widehat\ot H^T
 =
 \mathcal J_{\boldsymbol\alpha',\boldsymbol\gamma'}^T
 (F^T\ot H^T)
 (\mathcal J_{\boldsymbol\alpha,\boldsymbol\gamma}^T)^{-1}
 =(F\ot H)^T.
\end{equation}
Because \eqref{eq:oriented-staircase} is pairwise, the cross tensorators
satisfy
\[
\mathcal J_{\boldsymbol\alpha\boldsymbol\gamma,\boldsymbol\delta}^T
(\mathcal J_{\boldsymbol\alpha,\boldsymbol\gamma}^T\ot\id)
=
\mathcal J_{\boldsymbol\alpha,\boldsymbol\gamma\boldsymbol\delta}^T
(\id\ot\mathcal J_{\boldsymbol\gamma,\boldsymbol\delta}^T).
\]
Thus \eqref{eq:transported-tensor-product} is associative.  Applying
\eqref{eq:boundary-gauge-transport} to the canonical snake, hexagon, and
cap-sliding identities proves their transported versions.  In
particular, the two mixed crossings are categorical mates with respect
to \(\ev_n^T,\coev_n^T\); they are not ordinary unsigned matrix
transposes.

\subsection{Coherent root--center transport}
\label{sec:oriented-root-center}

The preceding gauge connects \(T_n\) with the untwisted braiding.  We now
include the tensorator of \(\mathcal F_{\rc}\) on the same oriented words.
Define the signed color
\[
 N(n^+)=n,\qquad N(n^-)=-n.
\]
If \(x\in X_\alpha\) has bidegree \(p=(p_1,p_2)\), then
\begin{equation}\label{eq:oriented-Cartan-weights}
\begin{aligned}
 x\tri\mathsf h_1
 &=(N(\alpha)(\beta+1)-p_2)x,&
 x\tri\mathsf h_2
 &=(-N(\alpha)\beta-p_1)x,\\
 x\tri\mathsf c_1
 &=-2N(\alpha)(\beta+1)x,&
 x\tri\mathsf c_2
 &=2N(\alpha)\beta x.
\end{aligned}
\end{equation}
For \(\alpha,\gamma\in\{n^+,n^-\}\), define
\begin{equation}\label{eq:oriented-F-exponent}
\begin{aligned}
\Gamma_{\alpha\mid\gamma}(p,q)
=\frac12\bigl[
&N(\gamma)((\beta+1)p_1-\beta p_2)\\
&+N(\alpha)(\beta q_2-(\beta+1)q_1)
\bigr].
\end{aligned}
\end{equation}
Let \(\mathcal F_{\rc}^{[r]}\) be the coherent \(r\)-fold tensorator of
the normalized twist and put
\begin{equation}\label{eq:oriented-F-tensorator}
 \mathsf F_{\boldsymbol\alpha}^{\mathrm{or}}
 =
 \rho_{X_{\boldsymbol\alpha}}(\mathcal F_{\rc}^{[r]}).
\end{equation}
Equations \eqref{eq:Frc} and \eqref{eq:oriented-Cartan-weights} give the
explicit diagonal action
\begin{equation}\label{eq:oriented-F-action}
\mathsf F_{\boldsymbol\alpha}^{\mathrm{or}}
(x_1\ot\cdots\ot x_r)
=
Q^{\sum_{i<j}
\Gamma_{\alpha_i\mid\alpha_j}(p_i,p_j)}
x_1\ot\cdots\ot x_r.
\end{equation}
The negative-color signs in this formula come from categorical duality,
not from an additional convention.

Let \((c_{\boldsymbol\alpha}^{\rc})_i\) be the braiding of the twisted
paired double.  Monoidal transport under a Drinfeld twist gives
\begin{equation}\label{eq:oriented-twisted-braiding}
 (c_{\boldsymbol\alpha}^{\rc})_i
 =
 (\mathsf F_{s_i\boldsymbol\alpha}^{\mathrm{or}})^{-1}
 (c_{\boldsymbol\alpha}^{\mathrm s})_i
 \mathsf F_{\boldsymbol\alpha}^{\mathrm{or}}.
\end{equation}
Define the combined gauge
\begin{equation}\label{eq:combined-oriented-gauge}
 \mathsf G_{\boldsymbol\alpha}^{\rc,\mathrm{or}}
 =
 G_{\boldsymbol\alpha}^{\mathrm{or}}
 \mathsf F_{\boldsymbol\alpha}^{\mathrm{or}}.
\end{equation}
Substitution in \eqref{eq:oriented-local-transport} yields
\begin{equation}\label{eq:oriented-combined-transport}
 (T_{\boldsymbol\alpha})_i
 =
 \mathsf G_{s_i\boldsymbol\alpha}^{\rc,\mathrm{or}}
 (c_{\boldsymbol\alpha}^{\rc})_i
 (\mathsf G_{\boldsymbol\alpha}^{\rc,\mathrm{or}})^{-1}.
\end{equation}
This is the fixed-color, all-orientations version of the positive braid
formula \eqref{eq:braid-root-comparison}.

The twist transports the canonical cups to
\begin{equation}\label{eq:root-center-cups}
\ev_n^{\rc}
=\ev_n\mathsf F_{(n^-,n^+)}^{\mathrm{or}},
\qquad
\coev_n^{\rc}
=(\mathsf F_{(n^+,n^-)}^{\mathrm{or}})^{-1}\coev_n.
\end{equation}
Transport once more by the combined gauge.  The twist tensorator cancels:
\begin{align}
\ev_n^{\rc}
(\mathsf G_{(n^-,n^+)}^{\rc,\mathrm{or}})^{-1}
&=\ev_n(G_{(n^-,n^+)}^{\mathrm{or}})^{-1}
=\ev_n^T,\label{eq:root-center-evaluation-cancels}\\
\mathsf G_{(n^+,n^-)}^{\rc,\mathrm{or}}\coev_n^{\rc}
&=G_{(n^+,n^-)}^{\mathrm{or}}\coev_n
=\coev_n^T.\label{eq:root-center-coevaluation-cancels}
\end{align}
Hence the root--center twist creates no second system of \(T\)-cups or
\(T\)-caps.  The same combined gauge transports the snake, hexagon,
cap-sliding, and categorical-mate identities.  This is the coherence
needed to pass from the braid comparison to oriented rigid tangles.

\subsection{Ordinary cups and partial transpose}

For \(y\) of bidegree \((a,b)\), define
\[
 \gamma_n(a,b)=Q^{-2ab+2n(\beta+1)a-2n\beta b},
\qquad
 A_n(y)=\gamma_n(a,b)^{-1}y.
\]
Then the transported and ordinary cups satisfy
\begin{equation}\label{eq:cup-gauge}
\ev_n^T=\ev_n(A_n^*\ot\id),\qquad
\coev_n^T=(\id\ot(A_n^*)^{-1})\coev_n.
\end{equation}
Every internal negatively oriented edge therefore carries \(A_n^*\) at
one end and \((A_n^*)^{-1}\) at the other.  These factors cancel edge by
edge, and no duality gauge remains on a \((1,1)\)-tangle with positive
boundary.

\begin{lemma}[Partial-transpose convention]\label{lem:partial-transpose}
Let \(\widetilde A^{\mathrm{GK}}\) be the matrix partial transpose used in
the normal long-knot state sum, and let
\(\widetilde A^{\mathrm{cat}}\) be the categorical mate obtained by
bending a strand through the super evaluation and coevaluation.  Under
the canonical source and target identifications,
\[
 \widetilde A^{\mathrm{cat}}
 =\widetilde{A^{-1}}^{\mathrm{GK}},\qquad
 \widetilde{\mu A}^{\mathrm{cat}}
 =\mu^{-1}\widetilde A^{\mathrm{cat}}.
\]
In particular, the partial transposes of both \(R'_{n,\beta}\) and
\((R'_{n,\beta})^{-1}\) are invertible.
\end{lemma}

\begin{proof}
Expand both sides using a homogeneous basis and the super cup and cap.
The inverse appears because the categorical mate bends a crossing whose
orientation is reversed.  A scalar attached to that crossing is therefore
inverted.  Invertibility follows independently for both partial
transposes because they are identified with mixed braidings in a rigid
braided category.
\end{proof}

\subsection{The normal long-knot operator}

Let \(D\) be a normal long-knot diagram and let \(J_R(D)\) be its
endomorphism-valued rigid-\(R\) state sum.  Composing the local comparison
at every crossing, cup, and cap gives the following global statement.

\begin{theorem}[Rigid transport]\label{thm:rigid}
\begin{equation}\label{eq:long-transport}
 \Phi_nJ_{T_n}(D)\Phi_n^{-1}
 =
 (\lambda_{n,\beta}^{\cen})^{\Writhe(D)}
 J_{R'_{n,\beta}}(D).
\end{equation}
If \(D\) is balanced, then \(\Writhe(D)=0\) and the scalar disappears.
For arbitrary writhe, the zero-framing normalization
\[
 \widehat J_R(D)
 =(\Theta_R^+)^{-\Writhe(D)}J_R(D)
\]
satisfies
\begin{equation}\label{eq:normalized-transport}
 \Phi_n\widehat J_{T_n}(D)\Phi_n^{-1}
 =\widehat J_{R'_{n,\beta}}(D).
\end{equation}
\end{theorem}

\begin{proof}
The coherent staircase and twist tensorators cancel at every internal
color word.  By \eqref{eq:cup-gauge}, the duality gauges cancel on every
internal negative edge.  \Cref{lem:partial-transpose} shows that the four
oriented local crossings have the same scalar rule: a positive crossing
contributes \(\lambda_{n,\beta}^{\cen}\), and a negative crossing
contributes its inverse.  This proves \eqref{eq:long-transport}.
A positive curl obeys
\[
 \Phi_n\Theta_{T_n}^+\Phi_n^{-1}
 =\lambda_{n,\beta}^{\cen}\Theta_{R'}^+,
\]
and the negative curl has the inverse scalar.  Substitution into the
definition of \(\widehat J\) proves \eqref{eq:normalized-transport}.
\end{proof}

\section{Identification with the colored Links--Gould invariant}
\label{sec:LG}

\subsection{Scalarity}

The operator \(J_{R'_{n,\beta}}(D)\) is an even
\(U_\hbar^{\Pi_{\ao}}(\mathfrak{sl}(2|1))\)-module endomorphism of
\(W_{n,\beta}^{\ao}\): braidings, their inverses, cups, caps, tensor
products, and compositions are all module morphisms.

\begin{proposition}\label{prop:Schur}
\[
 \End_{U_\hbar^{\Pi_{\ao}}(\mathfrak{sl}(2|1))}
 (W_{n,\beta}^{\ao})
 =\hK\,\id.
\]
Consequently, there is a unique scalar
\(LG'_{n,\mathcal K}(Q^{n\beta},Q)\) such that
\[
 J_{R'_{n,\beta}}(D)
 =LG'_{n,\mathcal K}(Q^{n\beta},Q)\id.
\]
\end{proposition}

\begin{proof}
An endomorphism sends the one-dimensional highest-weight space to itself.
It therefore acts on \(w_{n,\beta}\) by a scalar.  Since
\(w_{n,\beta}\) is cyclic by the proof of \cref{thm:object}, the same
scalar acts on all of \(W_{n,\beta}^{\ao}\).
\end{proof}

Combining \cref{thm:rigid,prop:Schur}, and recalling that
\(V_{n,\mathcal K}\) is the \((1,1)\)-entry of \(J_{T_n}(D)\), gives
\begin{equation}\label{eq:V-LGprime}
 V_{n,\mathcal K}(Q^{2n\beta+n},Q^2)
 =LG'_{n,\mathcal K}(Q^{n\beta},Q).
\end{equation}
In fact, the argument proves the stronger scalar identity
\[
 J_{T_n}(D)
 =LG'_{n,\mathcal K}(Q^{n\beta},Q)\id_{\mathcal Y_n}
\]
for balanced normal diagrams.  Thus scalarity is a theorem rather than
an input.

\subsection{Odd reflection and the distinguished Borel}

Set \(\alpha=n\beta\).  Reflect the all-odd simple system at the first
isotropic root.  The distinguished Cartan data are
\[
\gamma_1=\alpha_1+\alpha_2,\qquad
\gamma_2=-\alpha_1,\qquad
h_1^{\mathrm{dist}}=\mathsf h_1+\mathsf h_2,\qquad
h_2^{\mathrm{dist}}=-\mathsf h_1.
\]
If \(w=w_{n,\beta}\), typicality implies
\(w^{\mathrm{dist}}=w\tri G_1\neq0\).  After the right-to-left conversion,
\[
 h_1^{\mathrm{dist}}\cdot w^{\mathrm{dist}}=(n-1)w^{\mathrm{dist}},
\qquad
 h_2^{\mathrm{dist}}\cdot w^{\mathrm{dist}}=n\beta
 w^{\mathrm{dist}}.
\]
The new highest vector is odd, and hence
\begin{equation}\label{eq:odd-reflection-object}
 W_{n,\beta}^{\ao}\longleftrightarrow
 \Pi V_0(n-1,n\beta).
\end{equation}
The overall parity shift does not change the normalized
\((1,1)\)-tangle scalar: its possible parity-line sign is a writhe factor
and vanishes in the balanced or zero-framed normalization.

The independence of the Drinfeld--Jimbo quantization from the choice of
simple roots, together with the quantum-group/Kontsevich comparison for
Lie superalgebras of type A--G, identifies the rigid tangle scalar of
\eqref{eq:odd-reflection-object} with the standard colored Links--Gould
scalar \cite{GeerCartan,GeerKontsevich}.  What remains is the direction of
the braiding.  Standard Links--Gould conventions use left modules and
\(\tau^{\mathrm s}(\pi\ot\pi)(\mathcal R)\), whereas
\eqref{eq:right-braiding} uses \(\mathcal R^{-1}\).  The side switch
therefore reverses every crossing, equivalently replacing
\(Q\) by \(Q^{-1}\) or taking the mirror.  This choice is made at the
quantum-group level and is independent of \(n\).  It agrees with the
known \(n=2\) normalization \cite[Theorem~1.5]{ColoredLG}.

\begin{theorem}\label{thm:standard-comparison}
For every oriented knot \(\mathcal K\),
\begin{equation}\label{eq:LGprime-standard}
 LG'_{n,\mathcal K}(Q^{n\beta},Q)
 =
 LG_{\mathcal K}^{(n)}(Q^{-n\beta},Q^{-1})
 =
 LG_{\overline{\mathcal K}}^{(n)}(Q^{n\beta},Q).
\end{equation}
\end{theorem}

\begin{proof}
The odd reflection identifies the coloring object by
\eqref{eq:odd-reflection-object}.  The braided tensor equivalence between
the all-odd and distinguished quantizations transports braiding and
duality, so their normalized rigid \((1,1)\)-tangle scalars agree.
The right-to-left conversion changes the chosen braiding to its inverse;
this produces the simultaneous substitutions
\((Q^{n\beta},Q)\mapsto(Q^{-n\beta},Q^{-1})\).
Equivalently, keeping the parameters fixed replaces \(\mathcal K\) by
its mirror.  The normalization is fixed uniformly in \(n\) by the
right-module convention and agrees with the \(n=2\) comparison in
\cite{ColoredLG}.  This proves \eqref{eq:LGprime-standard}.
\end{proof}

\begin{proof}[Proof of \cref{thm:main}]
Substitute \eqref{eq:LGprime-standard} into \eqref{eq:V-LGprime}.
The scalar-identity assertion follows from \cref{prop:Schur}.
\end{proof}

\section{Concluding remarks}

The proof separates three effects that are difficult to distinguish at
the matrix level.  The gauge \(G_m^{\mathrm s}\) converts the Nichols
operator into a super Yetter--Drinfeld braiding.  The tensorator of
\(\mathcal F_{\rc}\) separates the root and central Hopf structures.
Finally, the one-dimensional central \(R^{-1}\)-factor contributes only
the scalar \(Q^{2n^2\beta(\beta+1)}\), which is removed by balanced
normal form or writhe normalization.

This separation also explains why scalarity holds.  It is not a special
feature of a chosen matrix entry: after the root--center decomposition,
the long-knot operator is an endomorphism of a simple typical
\(U_\hbar(\mathfrak{sl}(2|1))\)-module.  The original \((1,1)\)-entry is
therefore the scalar by which the whole operator acts.

\appendix

\section{First root degrees of the canonical tensor}

For reference, the inverses of the pairing matrices are
\[
\begin{aligned}
 M_{2r+1}^{-1}&=\frac{(-1)^r}{P_r(Q)}I_2,\\
 M_{2r}^{-1}&=\frac{(-1)^r}{P_r(Q)}
 \begin{pmatrix}1&Q^{-r}\\Q^{-r}&1\end{pmatrix}.
\end{aligned}
\]
The super right-dual convention multiplies degree \(m\) by \((-1)^m\).
Thus the signed root coevaluation begins with
\[
 1-x_1\ot F_1-x_2\ot F_2+\text{terms of root degree at least two}.
\]
Its inverse, the standard root factor of the universal \(R\)-matrix,
begins with
\[
 1+x_1\ot F_1+x_2\ot F_2+\text{terms of root degree at least two}.
\]
These signs are forced simultaneously by the right-dual evaluation and
the cross relation \(x_iF_i+F_ix_i=k_i-L_i\).

\section{Comparison with the Garoufalidis--Kashaev conventions}
\label{app:GK-conventions}

This appendix verifies that the operator denoted by \(T_n\) in the main
text is the operator of \cite[Theorem~3.6 and Definition~7.5]{GK}, rather
than merely an operator with the same dimension or parameter count.

\subsection{Parameters and the space \texorpdfstring{\(Y_n\)}{Yn}}
\label{app:GK-parameters}

To distinguish the notation of \cite{GK} from ours, add a superscript
\({\rm GK}\) to its parameters.  The required specialization is
\begin{equation}\label{eq:GK-dictionary}
 q^{\rm GK}=Q^2,\qquad
 q_{12}^{\rm GK}=q_{21}^{\rm GK}=-Q,\qquad
 t^{\rm GK}=Q^{2n\beta+n}=t.
\end{equation}
Indeed, the diagonal braiding determined by
\(q_{11}^{\rm GK}=q_{22}^{\rm GK}=-1\) and
\(q_{12}^{\rm GK}=q_{21}^{\rm GK}=-Q\) is exactly the full braiding
\(\chi\) of \eqref{eq:chi}.  The parametrization following
\cite[Proposition~7.4]{GK} therefore gives
\begin{equation}\label{eq:GK-t1-t2}
 \begin{aligned}
 t_1^{\rm GK}
 &=\frac{1}{(q^{\rm GK})^{n/2}t^{\rm GK}}
   =Q^{-2n(\beta+1)}=t_{1,n},\\
 t_2^{\rm GK}
 &=\frac{t^{\rm GK}}{(q^{\rm GK})^{n/2}}
   =Q^{2n\beta}=t_{2,n}.
 \end{aligned}
\end{equation}
In particular, \(t_1^{\rm GK}t_2^{\rm GK}(q^{\rm GK})^n=1\).
The coefficient of the invariant vector in \cite[(110)--(113)]{GK} is
\begin{equation}\label{eq:GK-alpha}
 \alpha=t_2^{\rm GK}(-q_{12}^{\rm GK})^n
 =Q^{2n\beta+n}=t.
\end{equation}
Consequently,
\[
 v_{n,\alpha}=(x_1x_2)^n+\alpha(x_2x_1)^n=v_n,
\]
and the basis in \cite[Proposition~7.4]{GK} becomes exactly the basis
\eqref{eq:Yn-basis}.  Thus, after the specialization
\eqref{eq:GK-dictionary} and extension of scalars to \(\hK\), there is a
canonical identification
\begin{equation}\label{eq:GK-space-identification}
 Y_n^{\rm GK}\otimes_{\mathbb F}\hK\cong\mathcal Y_n
\end{equation}
inside \(B_{\hK}\).  Since \(Q=e^\hbar\), the parameter \(q^{\rm GK}=Q^2\)
is not a root of unity in the formal setting.  Moreover,
\(\beta\ne-1,0\) gives respectively
\(t_1^{\rm GK}\ne1\) and \(t_2^{\rm GK}\ne1\), so the hypotheses of
\cite[Proposition~7.4]{GK} hold.

\subsection{The right action and the original definition of
\texorpdfstring{\(T_n\)}{Tn}}
\label{app:GK-action}

Let \(\varphi_n\) be the scaling automorphism of the braided Hopf algebra
\(B_{\hK}\) determined by
\begin{equation}\label{eq:GK-scaling}
 \varphi_n(x_i)=t_{i,n}x_i\qquad(i=1,2).
\end{equation}
Write
\[
 a\tri_B h=\lambda_R(a\ot h)
\]
for the twisted right-adjoint action of \cite[Theorem~3.5]{GK}.  On the
generators it is
\begin{equation}\label{eq:GK-right-action}
 a\tri_B x_i
 =ax_i-t_{i,n}\chi(a,x_i)x_i a,
\end{equation}
which is precisely the restriction used in \eqref{eq:Yn-action}.

Under \eqref{eq:GK-space-identification}, the operator called \(T_n\) in
\cite[Theorem~3.6 and Definition~7.5]{GK} is
\begin{equation}\label{eq:GK-T-definition}
 T_n^{\rm GK}
 =
 (\varphi_n\ot\lambda_R)
 (\tau_\chi\ot\id)
 (\id\ot\Delta_B)
 \big|_{\mathcal Y_n\ot\mathcal Y_n}.
\end{equation}
Thus, if \(a,b\) are homogeneous, \(\deg a=p\), and
\[
 \Delta_B(b)=\sum b^{(1)}\ot b^{(2)},\qquad
 \deg b^{(1)}=\rho,\quad \deg b^{(2)}=\sigma,
\]
then
\begin{equation}\label{eq:GK-T-Sweedler}
 T_n^{\rm GK}(a\ot b)
 =\sum \chi(p,\rho)\,
 \varphi_n(b^{(1)})\ot(a\tri_B b^{(2)}).
\end{equation}
This fixes both the tensor order and the right-action convention.

\subsection{Comparison with the super gauge}
\label{app:GK-super-gauge}

For a homogeneous \(a\in\mathcal Y_n\) of bidegree \((a_1,a_2)\),
\eqref{eq:Yn-action} and \eqref{eq:Un} give
\begin{equation}\label{eq:Psi-is-varphi}
 \Psi_n(a)=a\tri U_n
 =Q^{-2n(\beta+1)a_1+2n\beta a_2}a
 =t_{1,n}^{a_1}t_{2,n}^{a_2}a
 =\varphi_n(a).
\end{equation}
We now compare the two operators without suppressing their non-scalar
factors.  Let \(\deg a=p\), and use the notation of
\eqref{eq:GK-T-Sweedler}.  Since
\(G_2^{\mathrm s}=J_2^{\mathrm s}(\Psi_n\ot1)\), applying its inverse,
then \(c_n^{\mathrm s}\), and finally \(G_2^{\mathrm s}\) gives
\begin{align*}
 &G_2^{\mathrm s}c_n^{\mathrm s}(G_2^{\mathrm s})^{-1}
 (a\ot b)\\
 &\quad=
 \sum(-1)^{\bar p\bar\rho}
 r(\sigma,\rho)^{-1}r(p+\sigma,\rho)\,
 \Psi_n(b^{(1)})\ot(a\tri_B b^{(2)})\\
 &\quad=
 \sum\chi(p,\rho)\,
 \varphi_n(b^{(1)})\ot(a\tri_B b^{(2)})\\
 &\quad=T_n^{\rm GK}(a\ot b).
\end{align*}
Here the second equality uses \eqref{eq:Psi-is-varphi} and
\[
 (-1)^{\bar p\bar\rho}
 r(\sigma,\rho)^{-1}r(p+\sigma,\rho)
 =(-1)^{\bar p\bar\rho}r(p,\rho)=\chi(p,\rho).
\]
This proves that the \(T_n\) in \cref{prop:braid-gauge} is literally the
base-changed Garoufalidis--Kashaev operator under
\eqref{eq:GK-space-identification}; it is not a newly defined operator
known only to be gauge equivalent to it.

\section{A parameter dictionary}

\begin{center}
\begin{tabular}{@{}lll@{}}
\toprule
Symbol & Meaning & Relation\\
\midrule
\(Q\) & formal quantum parameter & \(Q=e^\hbar\)\\
\(\beta\) & highest-weight parameter & \(\beta\neq0,-1\)\\
\(\alpha\) & standard Links--Gould weight & \(\alpha=n\beta\)\\
\(z\) & Nichols color parameter & \(z=Q^{2\beta+1}\)\\
\(t\) & first \(V_n\)-parameter & \(t=z^n=Q^{2n\beta+n}\)\\
\(q\) & second \(V_n\)-parameter & \(q=Q^2\)\\
\(t_{1,n}\) & first central character & \(Q^{-2n(\beta+1)}\)\\
\(t_{2,n}\) & second central character & \(Q^{2n\beta}\)\\
\bottomrule
\end{tabular}
\end{center}

\section*{Acknowledgments}

The author is deeply grateful to his advisor, Nicolai Reshetikhin, for
carefully reviewing this article, offering valuable suggestions, and providing
encouragement and support.  He also warmly thanks his family, friends, and
partner for standing by him during times of discouragement.

The author used large language models as an aid in generating and refining
portions of the text.  The raw model-generated material was generally not of
publishable quality and was not used without substantial human intervention.
The present article was obtained only after the author repeatedly checked,
rewrote, and simplified that material.  The author takes full responsibility
for the content and for any remaining errors.

\end{document}